\documentclass[a4paper]{article}
\usepackage{fullpage}
\usepackage{amsmath}
\usepackage{amssymb}
\usepackage{amsthm}
\usepackage{graphicx}
\usepackage{bbm}
\usepackage{enumerate}
\usepackage{microtype}
\usepackage{tocbibind}
\usepackage{xcolor}
\usepackage[none]{hyphenat}
\usepackage[colorlinks=true,urlcolor=blue,linkcolor=blue,plainpages=false,pdfpagelabels]{hyperref}
\allowdisplaybreaks

\newtheorem{theorem}{Theorem}[section]

\newtheorem{lemma}[theorem]{Lemma}

\newtheorem{definition}[theorem]{Definition}

\newcommand{\vp}{\varphi}

\newcommand{\De}[1]{\ensuremath{{\Delta_{#1}}}}

\title{Normal forms and representable functions in Moisil logic}
\author{Andrei Sipo\c s${}^{a,b,c}$\\[2mm]
\footnotesize ${}^a$Research Center for Logic, Optimization and Security (LOS), Department of Computer Science,\\
\footnotesize Faculty of Mathematics and Computer Science, University of Bucharest,\\
\footnotesize Academiei 14, 010014 Bucharest, Romania\\[1mm]
\footnotesize ${}^b$Simion Stoilow Institute of Mathematics of the Romanian Academy,\\
\footnotesize Calea Grivi\c tei 21, 010702 Bucharest, Romania\\[1mm]
\footnotesize ${}^c$Institute for Logic and Data Science,\\
\footnotesize Popa Tatu 18, 010805 Bucharest, Romania\\[2mm]
\footnotesize Email: andrei.sipos@fmi.unibuc.ro\\
}
\date{}

\begin{document}

\maketitle

\begin{center}
{\em Dedicated to Professor George Georgescu on the occasion of his 77th birthday}
\end{center}

\vspace*{2mm}

\begin{abstract}
In this note, we determine, by a disjunctive normal form theorem, which functions on the standard $n$-nuanced \L ukasiewicz-Moisil algebra are representable by formulas and we show how this result may help in establishing the structure of the free algebras in this class.


\end{abstract}

\section{Introduction}

In the 1920s, Jan \L ukasiewicz introduced a three-valued logic \cite{Luk20} which he later generalized to $n$- and $\infty$-valued logics which now bear his name. Grigore C. Moisil was the first to attempt an algebrization of this kind of logic, when, in 1941, he introduced what he called $3$- and $4$-valued {\em \L ukasiewicz algebras} -- nowadays called {\em \L ukasiewicz-Moisil algebras} or simply {\em Moisil algebras} -- and later generalized them to the $n$-valued \cite{Moisil41} and the $\infty$-valued case (see \cite{Moisil1972}). In 1956, Alan Rose (in personal communication to R. Cignoli, see \cite[p. 2]{Cig70}) showed that this class of algebras is inadequate for \L ukasiewicz $n$-valued logic, since, for any $n \geq 4$, the set
$$\left\{0, \frac1n, \frac{n-1}n, 1\right\}$$
is a Moisil subalgebra of the standard $(n+1)$-valued Moisil algebra, which has the underlying set
$$L_{n+1}:=\left\{0,\frac1n,\ldots,\frac{n-1}n, 1\right\},$$
but is not closed under the \L ukasiewicz implication, which may be defined, for any $x$, $y \in [0,1]$, by
$$x \to y := \min(1,1-x+y),$$
since
$$\frac{n-1}n \to \frac1n = \frac2n.$$
An alternative was devised by C. C. Chang, who introduced in 1958 \cite{Chang1958} the now well-known class of {\em MV-algebras}. However, one can argue that \L ukasiewicz-Moisil algebras may still be considered algebras of logic, albeit for a different one, which is nowadays dubbed {\em Moisil logic}. A relatively short (but somewhat dated) introduction to \L ukasiewicz-Moisil algebras is \cite{Cig70}, while an exhaustive monograph from the early 1990s is \cite{BFGR}. More recent developments may be found in \cite{IL-det,DDILR,DDILG}.

The result presented above raises the question: for an $r \geq 1$, which $r$-ary functions on the standard Moisil algebra are representable by Moisil formulas? This question is answered by the main result, Theorem~\ref{nform}. In proving this theorem, we shall essentially provide a `disjunctive normal form' of sorts for formulas in Moisil logic. Using this result, we then study in a concrete way the structure of free Moisil algebras.

\section{Main results}

We shall fix a natural number $n \geq 2$.

\begin{definition}
A {\bf De Morgan algebra} is a tuple
$$\mathcal{L} = (L,\vee,\wedge,N,0,1)$$
such that $(L,\vee,\wedge,0,1)$ is a bounded distributive lattice and, for any $x$, $y\in L$,
\begin{enumerate}[(i)]
\item $NNx=x$;
\item $N(x \vee y) = Nx \wedge Ny$.
\end{enumerate}
\end{definition}

\begin{definition}
An {\bf $n$-nuanced} or {\bf $(n+1)$-valued Moisil algebra}, also called an {\bf $LM_n$-algebra}, is a tuple
$$\mathcal{L} = (L,\vee,\wedge,N,0,1,\De{1},\ldots,\De{n})$$
such that $(L,\vee,\wedge,N,0,1)$ is a De Morgan algebra and $\De{1},\ldots,\De{n}$ are unary operations on $L$ (called {\bf nuances} or {\bf Chrysippian endomorphisms}) such that, for any $i$, $j\in\{1,\ldots,n\}$ and any $x$, $y\in L$,
\begin{enumerate}[(i)]
\item $\De{i}(x\vee y) = \De{i}(x) \vee \De{i}(y)$;
\item $\De{i}(x) \vee N \De{i}(x) = 1$;
\item $\De{i} \De{j}(x) = \De{j}(x)$;
\item $\De{i} N x = N \De{n+1-i}(x)$;
\item if $i \leq j$, then $\De{i}(x) \leq \De{j}(x)$;
\item if, for all $k\in\{1,\ldots,n\}$, $\De{k}(x) = \De{k}(y)$, then $x = y$.
\end{enumerate}
\end{definition}

The standard $LM_n$-algebra, denoted by $\mathcal{L}_n$, has 
$$L_{n+1}=\left\{0,\frac1n,\ldots,\frac{n-1}n, 1\right\}$$
as its underlying set (as mentioned in the Introduction), its De Morgan algebra structure is the obvious one, and for all $i \in \{1,\ldots,n\}$, $j \in \{0,\ldots,n\}$, we have that
$$\Delta_i\left(\frac{j}{n}\right) = \begin{cases} 1, &\mbox{if  $i+j \geq n+1$,} \\ 
0, & \mbox{otherwise}. \end{cases}$$

{\it Moisil's representation theorem} (see, e.g., \cite[p. 25]{Cig70}) shows, in particular, that for any $LM_n$-algebra $A$ and any $a$, $b \in A$, there is a morphism $h : A \to \mathcal{L}_n$ such that $h(a)\neq h(b)$.

By property (v) in the definition, the nuances are progressively `contained' each one in the next, formalizing the intuitition that, if $\vp$ is a formula, then $\Delta_i\vp$ should `mean' that $\vp$ has the `truth value' greater or equal to $(n-1-i)/n$. A natural question that we may consider is whether we replace these operations by `independent' nuances. Such \emph{mutually exclusive nuances} $J_0,\ldots,J_n$ were introduced in \cite{Cig82} and later used in \cite{IL-det,DDILR} in order to obtain an alternative formulation and equational axiomatization of Moisil algebras. They are expressible over the signature of $LM_n$-algebras, as they are defined, for any $x$ in an algebra and any $i \in \{1,\ldots,n-1\}$, by
$$J_i(x) := \Delta_{n-i+1}(x) \wedge N\Delta_{n-i}(x),$$
and by $J_n:=\Delta_1$ and $J_0(x):=N \Delta_n(x)$. These $n+1$ terms have the following property when instantiated in $\mathcal{L}_n$: for all $i$, $j \in \{0,\ldots,n\}$, we have, denoting by $\delta_{ij}$ the Kronecker delta, that
\begin{equation}\label{j-eq}
J_i\left(\frac{j}{n}\right) = \delta_{ij}.
\end{equation}

We shall also denote, for any set $X$, by $T_n(X)$ the term algebra with variables from $X$ over the signature of $LM_n$-algebras. Note that this is not a $LM_n$-algebra, though it is a free algebra in the larger category of that signature.

\begin{definition}
Let $r\geq 1$ and $f : L_{n+1}^r \to L_{n+1}$. We say that $t \in T_n(\{x_1,\ldots,x_r\})$ is a {\bf representing term} for $f$ if for each morphism $\vp : T_n(\{x_1,\ldots,x_r\}) \to \mathcal{L}_n$ we have that
$$f(\vp(x_1),\ldots,\vp(x_r))=\vp(t).$$
We say that $f$ is {\bf Moisil representable} if it has a representing term.
\end{definition}

We may now present our main result.

\begin{theorem}\label{nform}
Let $r \geq 1$ and $f : L_{n+1}^r \to L_{n+1}$. TFAE:
\begin{enumerate}[(a)]
 \item $f$ is Moisil representable;
 \item for any $a_1,\ldots,a_r \in L_{n+1}$ we have that $f(a_1,\ldots,a_r) \in \{0,1,a_1,\ldots,a_r,1-a_1,\ldots,1-a_r\}$.
\end{enumerate}
\end{theorem}

\begin{proof}
We first prove `$(a)\Rightarrow(b)$'. Let $a_1,\ldots,a_r \in L_{n+1}$. Set $M:=\{0,1,a_1,\ldots,a_r,1-a_1,\ldots,1-a_r\}$. Observe that $M$ is a subalgebra of $\mathcal{L}_n$, and denote by $\mu$ the inclusion morphism. Let $\vp' : T_n(\{x_1,\ldots,x_r\}) \to M$ be the unique morphism such that for all $i$, $\vp'(x_i)=a_i$ and set $\vp := \mu \circ \vp'$. Let $t$ be a representing term for $f$. We have that
$$f(a_1,\ldots,a_r) = f(\vp(x_1),\ldots,\vp(x_r)) = \vp(t) = \mu(\vp'(t)) \in M.$$

We now prove `$(b)\Rightarrow(a)$'. For each $a_1,\ldots,a_r$, $a \in L_{n+1}$ such that $a \in \{0,1,a_1,\ldots,a_r,1-a_1,\ldots,1-a_r\}$, we define the following term belonging to $T_n(\{x_1,\ldots,x_r\})$:
\[
 s(a_1,\ldots,a_r,a) :=
  \begin{cases} 
      \hfill 1, \hfill &  \text{if } a=1,\\ 
      \hfill 0, \hfill &  \text{if } a=0,\\ 
      \hfill x_i, \hfill &  \text{if } a\notin\{0,1\} \text{ and } i:=\min\{j \mid a=a_j\},\\ 
      \hfill Nx_i, \hfill &  \text{if } a\notin\{0,1,a_1,\ldots,a_r\}\text{ and } i:=\min\{j \mid a=1-a_j\}.
  \end{cases}
\]
We define the term (where, for each $i$, $na_i$ denotes the product of $n$ and $a_i$, which is always a natural number $\leq n$)
$$t:= \bigvee_{(a_1,\ldots,a_r) \in L_{n+1}^r} \big( J_{na_1}(x_1) \wedge\ldots\wedge J_{na_r}(x_r) \wedge s(a_1,\ldots,a_r,f(a_1,\ldots,a_r))\big).$$
We now prove that $t$ is a representing term for $f$. Let $\vp : T_n(\{x_1,\ldots,x_r\}) \to \mathcal{L}_n$ be a morphism. Set for all $i$, $b_i := \vp(x_i)$ and then $b:=f(b_1,\ldots,b_r)$. What we must show is that $\vp(t)=b$.

Since $\vp$ is a morphism, we have that
$$\vp(t) = \bigvee_{(a_1,\ldots,a_r) \in L_{n+1}^r} \big( J_{na_1}(b_1) \wedge\ldots\wedge J_{na_r}(b_r) \wedge \vp(s(a_1,\ldots,a_r,f(a_1,\ldots,a_r)))\big),$$
a disjunction that, by \eqref{j-eq}, collapses to the term indexed by $(b_1,\ldots,b_r)$, i.e. we have that
$$\vp(t) = \vp(s(b_1,\ldots,b_r,b)).$$

We focus on the case where $b \notin \{0,1\}$ and there is an $i$ such that $i=\min\{j \mid b=b_j\}$, so $b=b_i$ (the other cases are treated similarly). Then $s(b_1,\ldots,b_r,b)=x_i$ and therefore
$$\vp(t) = \vp(x_i) = b_i = b$$
and we are done.
\end{proof}

As pointed out in the Introduction, the proof above gives a canonical representing term for each $f$, a `disjunctive normal form' of sorts. We shall use this result in order to study the structure of free $LM_n$-algebras. For that, denote, for any $r \geq 1$, by $R_{n,r}$ the set of all Moisil representable functions $f : L_{n+1}^r \to L_{n+1}$, considered as a $LM_n$-algebra with the operations defined componentwise. We have the following lemma.

\begin{lemma}
Let $r \geq 1$. Denote, for each $i \in \{1,\ldots,r\}$, by $p_i : L_{n+1}^r \to L_{n+1}$ the canonical projection on the $i$th argument. Then $R_{n,r}$ is freely generated by these projections.
\end{lemma}

We omit the proof, since the bulk of its argument is well known in universal algebra, since the only fact specific to $LM_n$-algebras that is used is that any algebra has enough morphisms into $\mathcal{L}_n$ to separate its elements (the standard algebra generates the whole variety, by Moisil's representation theorem).

Now, combining the two results above, we see that the free $LM_n$-algebra over $r$ generators can be most concretely described as the collection of all possible `truth tables' constructed using the rule in Theorem~\ref{nform}. Since that rule puts a constraint on each line in a table in an independent way (the constraint being that the result in the final column must be an element of the corresponding subalgebra), the free algebra is transparently isomorphic to a product:
$$F_n(r) \cong \prod_A A^{\alpha(r,A)},$$
where the product is taken over all subalgebras $A$ of $\mathcal{L}_n$ and $\alpha(r,A)$ is the number of all tuples $(a_1,\ldots,a_r)$ that generate $A$. We may refine this result a little. Consider, for simplicity, that $n$ is odd. Then, as shown e.g. in \cite{IL-det}, a subalgebra of $\mathcal{L}_n$ is completely determined by a choice of a subset of $\left\{\frac1n,\ldots,\frac{\frac{n-1}2}n\right\}$. For each $k \in \{1,\ldots,\frac{n+1}2\}$ the $(k-1)$-element subsets are therefore $\binom{\frac{n-1}2}{k-1}$ in number, and we denote the corresponding $2k$-element subalgebras of $\mathcal{L}_n$, for each $j \in \{1,\ldots,\binom{\frac{n-1}2}{k-1}\}$, by $A_{k,j}$. The relation above becomes:
$$F_n(r) \cong \prod_{k=1}^{\frac{n+1}2} \prod_{j=1}^{\binom{\frac{n-1}2}{k-1}} A_{k,j}^{\alpha(r,k)},$$
where $\alpha(r,k)$ is the number of all tuples $(a_1,\ldots,a_r)$ that generate a $2k$-element subalgebra, since that number does not depend on the specific such subalgebra (and not even on $n$). Cignoli arrives, albeit by a different, more abstract route, using the structure of finite $LM_n$-algebras, at the same result in \cite{Cig70}, where he derives the formula
$$\alpha(r,k) = 2^r \sum_{i=0}^{k-1} (-1)^i \binom{k-1}i (k-i)^r.$$
He then proceeds to obtain a similar result for the case of even $n$, using arguments a bit more intricate, but not qualitatively different. Our method, we think, has the advantage of giving a concrete feel for how one may work with such an algebra by providing a tool of a constructive flavour, namely the disjunctive normal form result obtained above.

\section{Acknowledgements}

I would like to thank Ioana Leu\c stean for originally telling me about this problem in January 2017. Even though the problem got solved almost immediately, the result was not made public for various reasons, a situation which this note aims to correct.

This work was supported by a grant of the Romanian National Authority for Scientific Research and Innovation, CNCS-UEFISCDI, project number PN-II-RU-TE-2014-4-0730.

\end{document}